\documentclass[a4paper,reqno,11pt]{amsart}
\usepackage[margin=1.22in]{geometry} 
\usepackage[dvipsnames]{xcolor}

\usepackage{graphicx}
\usepackage[utf8]{inputenc}
\usepackage{xcolor}
\usepackage{latexsym}
\usepackage{lipsum}
\usepackage{nicefrac}
\usepackage[english]{babel}

\usepackage{amsmath}
\usepackage{amssymb}
\usepackage{fge}
\usepackage{mathtools}
\usepackage{amsthm}
\usepackage{verbatim}
\usepackage{mathrsfs}
\usepackage[cal=boondox,scr=boondoxo]{mathalfa}
\usepackage{braket}
\usepackage{tikz-cd}
\usetikzlibrary{ positioning, shapes }
\usepackage[bookmarksopen=true]{hyperref}
\definecolor{darkgreen}{RGB}{0,80,0}
\hypersetup{
    colorlinks = true,
    linkbordercolor = {white},
    linkcolor = {NavyBlue},
    anchorcolor = {black},
    citecolor = {NavyBlue},
    filecolor = {cyan},
    menucolor = {NavyBlue},
    runcolor = {cyan},
    urlcolor = {NavyBlue}
}

\usepackage[noabbrev]{cleveref}
\usepackage[font={small}]{caption}
\usepackage{verbatim}
\usepackage{mathrsfs}
\usepackage[cal=boondox,scr=boondoxo]{mathalfa}
\usepackage[shortlabels]{enumitem}

\usepackage[utf8]{inputenc}

\usepackage{nicematrix}
\usetikzlibrary{patterns}

\newtheoremstyle{teoremas}
{12pt}
{13pt}
{\itshape}
{}
{\bfseries}
{}
{.5em}
{}

\theoremstyle{teoremas}
\newtheorem{theorem}{Theorem}[section]

\newtheorem{lemma}[theorem]{Lemma}

\newtheoremstyle{definition}
{12pt}
{12pt}
{}
{}
{\bfseries}
{}
{.5em}
{}

\theoremstyle{definition}

\newtheorem{example}[theorem]{Example}
\newtheorem{remark}[theorem]{Remark}
\newtheorem{claim}[theorem]{Claim}

\numberwithin{equation}{section}
\definecolor{details}{RGB}{0,0,255}
\definecolor{task}{RGB}{0,191,0}
\definecolor{sketch}{RGB}{255,0,0}

\renewcommand{\min}{\operatorname{min}}

\title[Preservation of log-concavity on gamma polynomials]{Preservation of log-concavity\\ on gamma polynomials}

\author{Luis Ferroni, Greta Panova, \and Lorenzo Venturello}
\address{(L. Ferroni)
 Institute for Advanced Study, Princeton (NJ), United States\newline\indent
 University of Pisa, Pisa, Italy.
}
\email{ferroni@ias.edu}
\email{luis.ferroni@unipi.it}
\address{(G. Panova)
 University of Southern California, Los Angeles (CA), United States
}
\email{gpanova@usc.edu}
\address{(L. Venturello)
 Universit\`a di Siena, Siena, Italy
}
\email{lorenzo.venturello@unisi.it}



\begin{document}

\begin{abstract}
    Every symmetric polynomial $h(x)$ with center of symmetry $n/2$ can be expressed as a linear combination in the basis $x^i(1+x)^{n-2i}$. The $\gamma$-polynomial of $h(x)$, which we denote $\gamma_h(x)$, records the coefficients of this linear combination. Two decades ago, Br\"and\'en and Gal independently showed that if $\gamma_h(x)$ has nonpositive real roots only, then so does $h(x)$. More recently, Br\"and\'en, Ferroni, and Jochemko proved using Lorentzian polynomials that if $\gamma_h(x)$ is ultra log-concave, then so is $h(x)$, and they raised the question of whether a similar statement can be proved for the usual notion of log-concavity. The purpose of this article is to show that the answer to the question of Br\"and\'en, Ferroni, and Jochemko is affirmative. One of the crucial ingredients of the proof is an inequality involving binomial numbers that we establish via a path-counting argument.
\end{abstract}

\maketitle

\section{Introduction}

A recurring endeavor in combinatorics consists in proving that special polynomials satisfy the properties of being unimodal or log-concave. Precisely, a polynomial $h(x) := h_0 + h_1 \,x + \cdots + h_n \,x^n$ with nonnegative coefficients is said to be \emph{unimodal} if there exists some index $i$ such that 
    \[h_0 \leq \cdots \leq h_{i-1} \leq h_i \geq h_{i+1} \geq \cdots \geq h_n,\]
whereas it is said to be \emph{log-concave} if, for each $0 < i < n$, the inequalities $h_i^2 \geq h_{i-1}\, h_{i+1}$ hold true. If the sequence $h_0,\ldots, h_n$ has no internal zeros, i.e., if $h_ih_j\neq 0$ for $i<j$ implies that $h_k\neq 0$ for $i\leq k\leq j$, then the property of being log-concave implies that of being unimodal. For detailed surveys on unimodality and log-concavity, we refer to \cite{stanley-unimodality,brenti,branden}.

If the polynomial $h(x)$ satisfies the conditions $h_i = h_{n-i}$ for each $i$, then we say that $h$ is \emph{symmetric} with center of symmetry $n/2$. We note that no assumptions on the degree of $h(x)$ have been made, thus in particular one is allowed to have $h_n=0$. Also, notice that the center of symmetry of a symmetric polynomial is unique.

An important tool to prove that such a symmetric polynomial is unimodal consists of proving that it is $\gamma$-positive. Precisely, symmetric polynomials with center of symmetry $n/2$ can be rewritten as follows:
    \[ h(x) = \sum_{i=0}^{\lfloor n/2\rfloor} \gamma_i\, x^i (1+x)^{n-2i},\]
where $\gamma_0,\ldots\gamma_{\lfloor n/2\rfloor}$ are uniquely determined.
Whenever all the coefficients $\gamma_i$ are nonnegative, then $h(x)$ is said to be \emph{$\gamma$-positive}. The \emph{$\gamma$-polynomial} associated to $h(x)$ is defined by $\gamma_h(x) := \gamma_0 + \gamma_1 x+\cdots + \gamma_{\lfloor n/2\rfloor} x^{\lfloor n/2\rfloor}$. It is not difficult to show that $\gamma$-positivity implies unimodality. For a thorough exposition on $\gamma$-positivity, we refer the reader to \cite{athanasiadis-gamma-positivity}.

Two decades ago, Br\"and\'en \cite{branden-gamma} and Gal \cite{gal} showed independently the following result.

\begin{theorem}[\cite{branden-gamma,gal}]\label{thm:real-rooted-preserved}
    Let $h(x)$ be a symmetric polynomial. The polynomial $\gamma_h(x)$ has only nonpositive real zeros if and only if so does $h(x)$.
\end{theorem}

The proof of the above result resorts to basic manipulations of (complex) univariate polynomials. This has been used extensively throughout the literature to approach various problems on unimodality, log-concavity, and real-rootedness.

More recently, Br\"and\'en, Ferroni, and Jochemko \cite[Theorem~3.3]{branden-ferroni-jochemko} proved a version of the above statement, for ultra log-concavity instead. Specifically, the polynomial $h(x) = h_0 + \cdots + h_n\,x^n$ is said to be \emph{ultra log-concave of order $m\geq n$} if the sequence \[\frac{h_0}{\binom{m}{0}}, \frac{h_1}{\binom{m}{1}}, \ldots, \frac{h_n}{\binom{m}{n}}\]
is log-concave. Their result can be reformulated as follows.

\begin{theorem}[\cite{branden-ferroni-jochemko}]\label{thm:ulc-preserved}
    Let $h(x)$ be a symmetric polynomial with center of symmetry $n/2$. If $\gamma_h(x)$ is ultra log-concave of order $\lfloor n/2\rfloor$ with no internal zeros, then $h(x)$ is ultra log-concave of order $n$ with no internal zeros.
\end{theorem}

In contrast to Theorem~\ref{thm:real-rooted-preserved}, in the last theorem the converse is false, see \cite[Example~3.4]{branden-ferroni-jochemko}. The proof of Theorem~\ref{thm:ulc-preserved} relies on properties of the class of Lorentzian polynomials introduced by Br\"and\'en and Huh \cite{branden-huh}. A question raised by Br\"and\'en, Ferroni, and Jochemko in their work is whether it was possible to prove a similar property but for the weaker notion of log-concavity. We show that the answer to their question is affirmative. The main result of this article is the following.

\begin{theorem}\label{thm:main}
    Let $h(x)$ be a symmetric polynomial. If $\gamma_h(x)$ is log-concave and has no internal zeros, then $h(x)$ is log-concave and has no internal zeros.
\end{theorem}

Once again, and via the same example of Br\"and\'en, Ferroni, and Jochemko \cite[Example~3.4]{branden-ferroni-jochemko}, the converse is not true. 

A reformulation of Theorem~\ref{thm:ulc-preserved} is that the operator on bivariate homogeneous polynomials appearing in \cite[Lemma~3.2]{branden-ferroni-jochemko} preserves the Lorentzian property. Lorentzian polynomials generalize the notion of ultra log-concavity whilst \emph{denormalized} Lorentzian polynomials generalize the notion of log-concavity. In particular, Theorem~\ref{thm:main} says that the aforementioned operator preserves the denormalized Lorentzian property too. As it is the case in the present situation, proving the preservation of the denormalized Lorentzian property is often considerably harder. One reason is that Lorentzian polynomials satisfy much better properties than denormalized Lorentzian---for example, nonnegative linear changes of variables preserve the Lorentzian property but often fail to preserve the denormalized Lorentzian property (see \cite[Theorem~2.10]{branden-huh} and \cite[Section~6]{branden-ferroni-jochemko}). In particular, in order to prove a result like Theorem~\ref{thm:main}, some ad-hoc strategy needs to be applied. In the present article, the proof of Theorem~\ref{thm:main} is essentially elementary, but relies on certain technical combinatorial identities and inequalities involving binomial coefficients and injections on lattice paths (see \Cref{lem: negatives on the right} and \Cref{thm: sum nonnegative}).

\section{The set up}

Throughout the paper a symmetric polynomial $h(x) = h_0 + h_1\, x + \cdots + h_n\, x^n$ with center of symmetry $n/2$ is given.  As mentioned earlier, we do not require $h_n\neq 0$, so that the degree of $h(x)$ is allowed to be less than $n$. Correspondingly, we will denote $\gamma_h(x) = \gamma_0 + \gamma_1\, x + \cdots + \gamma_{\lfloor n/2\rfloor}\, x^{\lfloor n/2\rfloor}$ the $\gamma$-polynomial associated to $h(x)$. 

The identity $h(x) = \sum_{j=0}^{\lfloor n/2\rfloor} \gamma_j\, x^j(1+x)^{n-2j}$ allows us to establish a relation between the $h_i$'s in terms of the $\gamma_j$'s. More precisely, we have:
\begin{equation}\label{eq:h-in-terms-of-gamma}
    h_i = \sum_{j=0}^i \binom{n-2j}{i-j}\gamma_j,
\end{equation}
for each $0\leq i\leq n$. Each quantity $h_i^2-h_{i-1}h_{i+1}$ for $0 < i < n$ can be regarded as a homogeneous polynomial of degree $2$ in the $\gamma_j$'s. However, when written in the basis given by monomials in the $\gamma_j$'s, this homogeneous polynomial may have negative coefficients. For concreteness, we include the following example.

\begin{example}\label{ex:gamma-degree-3}
    Consider $n=6$, so that $h(x) = h_0 + h_1\, x +\cdots + h_6\, x^6$ and, correspondingly, $\gamma_h(x) = \gamma_0 + \gamma_1\, x + \gamma_2\, x^2 + \gamma_3\, x^3$. We have the following identities:
    \begin{align*}
        h_0 = h_6 &= \gamma_0,\\
        h_1 = h_5 &= 6\,\gamma_0 + \gamma_1,\\
        h_2 = h_4 &= 15\,\gamma_0 + 4\,\gamma_1+\gamma_2,\\
        h_3 &= 20\,\gamma_0 + 6\,\gamma_1 + 2\,\gamma_2 + \gamma_3.
    \end{align*}
    From this, one may compute explicitly:
    \begin{align*}
        h_1^2 - h_0h_2 &= 21\,\gamma_0^2 + 8\, \gamma_0\gamma_1 + \gamma_1^2 - \gamma_0\gamma_2,\\
        h_2^2 - h_1h_3 &= 105\, \gamma_0^2 + 64\,\gamma_0\gamma_1 + 10\,\gamma_1^2 + 18\,\gamma_0\gamma_2 + 6\,\gamma_1\gamma_2 + \gamma_2^2 - 6\,\gamma_0\gamma_3 - \gamma_1\gamma_3,\\
        h_3^2 - h_2h_4 &=  175\,\gamma_0^2 + 120\,\gamma_0\gamma_1 + 20\,\gamma_1^2 + 50\,\gamma_0\gamma_2 + 16\,\gamma_1\gamma_2 + 3\,\gamma_2^2 + 40\,\gamma_0\gamma_3 +\\ &\qquad 12\,\gamma_1\gamma_3 + 4\,\gamma_2\gamma_3 + \gamma_3^2,
    \end{align*}
    and symmetrically, $h_4^2-h_3h_5 = h_2^2-h_1h_3$ and $h_5^2-h_6h_4 = h_1^2-h_0h_2$. Observe the presence of negative signs in the first and the second expression above.
\end{example}

The following is an elementary but important lemma that we will use throughout the proof.

\begin{lemma}\label{lemma:log-concavity-general}
    Let $a_0,\ldots,a_n$ be a sequence of nonnegative numbers without internal zeros. The sequence $a_0,\ldots,a_n$ is log-concave if and only if $a_ia_{j-1}\geq a_{i-1}a_{j}$ for all $1\leq i\leq j\leq n$. 
\end{lemma}

\begin{proof}
    The inequalities $a_i^2\geq a_{i-1}a_{i+1}$ are equivalent to: $\frac{a_i}{a_{i-1}} \geq \frac{a_{i+1}}{a_i}$.
    In other words, the log-concavity of the sequence $a_0,\ldots,a_n$ is equivalent to the sequence $b_1,\ldots, b_{n}$, where $b_i := \frac{a_i}{a_{i-1}}$, being weakly decreasing. That is to say that $b_i \geq b_j$ for $i\leq j$ which, in terms of the original sequence, reads $\frac{a_i}{a_{i-1}}\geq \frac{a_j}{a_{j-1}}$
    for all $i\leq j$. The result follows clearing denominators.
\end{proof}

\begin{example}\label{example:gamma-4}
    Consider $n=8$. The computation of $h_3^2-h_2h_4$ in terms of the $\gamma_j$'s is the following:
    \begin{align*} 
    h_3^2-h_2h_4 &=  1176\,\gamma_0^2 + 700\,\gamma_0\gamma_1 + 105\,\gamma_1^2 + 210\,\gamma_0\gamma_2 + 64\,\gamma_1\gamma_2 + 56\,\gamma_0\gamma_3 + \\
    &\qquad + 10\,\gamma_2^2 + 18\,\gamma_1\gamma_3 + 6\,\gamma_2\gamma_3 + \gamma_3^2 - 28\,\gamma_0\gamma_4 - 6\,\gamma_1\gamma_4 - \gamma_2\gamma_4\\&=1176\,\gamma_0^2 + 700\,\gamma_0\gamma_1 + \left[105\,(\gamma_1^2-\gamma_0\gamma_2) + 315\,\gamma_0\gamma_2\right]\\
    &\qquad+ \left[64\,(\gamma_1\gamma_2-\gamma_0\gamma_3) + 120\,\gamma_0\gamma_3\right] + 
    \left[10\,(\gamma_2^2-\gamma_1\gamma_3) + 28\,(\gamma_1\gamma_3-\gamma_0\gamma_4)\right]\\
    &\qquad + \left[6\,(\gamma_2\gamma_3 - \gamma_1\gamma_4)\right] +\left[\gamma_3^2 - \gamma_2\gamma_4\right].
    \end{align*}
    By Lemma~\ref{lemma:log-concavity-general}, if the sequence $\gamma_0,\gamma_1,\gamma_2,\gamma_3,\gamma_{4}$ is log-concave without internal zeros, then all the terms in the last expansion are nonnegative.  
\end{example}

The last example helps us outline the idea of the main proof. Notice that inside each of the square brackets all the terms $\gamma_j\gamma_k$ appearing (with either positive or negative sign) have fixed sum $j+k$. It is tempting to conjecture that $h_i^2-h_{i-1}h_{i+1}$ will in fact always be a nonnegative combination of terms of the form $\gamma_j\gamma_k - \gamma_{j'}\gamma_{k'}$ where $j+k=j'+k'$ and $|j-k| < |j'-k'|$. This is what we will prove in the next sections. 

\section{The first step of the proof}

In general, it is possible to describe the coefficient of each monomial $\gamma_j\gamma_k$ in $h_i^2-h_{i-1}h_{i+1}$ explicitly. From now on, we define
    \[ c_{jk}^{(i)} := [\gamma_j\gamma_k] \left(h_i^2-h_{i-1}h_{i+1}\right)\]
where, as customary in combinatorics, the square brackets mean that we extract the coefficient of $\gamma_j\gamma_k$ in the expression $h_i^2-h_{i-1}h_{i+1}$ regarded as a polynomial in the $\gamma_j$'s. In other words:

\[
h_i^2-h_{i-1}h_{i+1} = \sum_{j\leq k} c^{(i)}_{jk} \gamma_{j}\gamma_{k}.
\]	

It is immediate from \eqref{eq:h-in-terms-of-gamma}  that $c^{(i)}_{jk}=0$ whenever $k > i+1$. In other words, the nonvanishing monomials in $h_i^2-h_{i-1}h_{i+1}$ involve the variables $\gamma_0,\ldots,\gamma_{i+1}$ only.

\begin{lemma}\label{lemma:coeffs}
    The coefficients $c_{jk}^{(i)}$ can be computed explicitly as follows:
    \[ c_{jk}^{(i)} = \begin{cases} \binom{n-2j}{i-j}^2-\binom{n-2j}{i-j-1}\binom{n-2j}{i-j+1} & \text{ if $j=k$},\\ 2\binom{n-2j}{i-j}\binom{n-2k}{i-k} - \binom{n-2j}{i-j-1}\binom{n-2k}{i-k+1}-\binom{n-2j}{i-j+1}\binom{n-2k}{i-k-1}& \text{ if $j< k$}.
    \end{cases}\]
\end{lemma}

\begin{proof}
    The statement follows from \eqref{eq:h-in-terms-of-gamma}. Indeed, if $j<k$ we have that 
    \[ [\gamma_j\gamma_k]h_i^2 = 2([\gamma_j]h_i)([\gamma_k]h_i)=2\binom{n-2j}{i-j}\binom{n-2k}{i-k},\]
    and \begin{align*}
    [\gamma_j\gamma_k]h_{i-1}h_{i+1}&=([\gamma_j]h_{i-1})([\gamma_k]h_{i+1})+([\gamma_k]h_{i-1})([\gamma_j]h_{i+1})\\&=\binom{n-2j}{i-j-1}\binom{n-2k}{i-k+1}+\binom{n-2j}{i-j+1}\binom{n-2k}{i-k-1}.    
    \end{align*} 
    If $j=k$ we have a slightly different expression:
     \[ [\gamma_j^2]h_i^2 = ([\gamma_j]h_i)^2=\binom{n-2j}{i-j}^2,\]
    and
    \[
        [\gamma_j^2]h_{i-1}h_{i+1}=([\gamma_j]h_{i-1})([\gamma_j]h_{i+1})=\binom{n-2j}{i-j-1}\binom{n-2j}{i-j+1}.   \qedhere
    \]
\end{proof}


\begin{remark}\label{remark:cjji positive}
    It is easy to see that the sequence $a_{i}=\binom{n-2j}{i-j}$ is log-concave for any $j\leq \lfloor n/2\rfloor$. In particular we have that 
    \[
        c^{(i)}_{jj}=\binom{n-2j}{i-j}^2-\binom{n-2j}{i-j-1}\binom{n-2j}{i-j+1}=a_i^2-a_{i-1}a_{i+1} \geq 0.   \]
    Moreover, using twice the identity $\binom{n-2j}{k}=\binom{n-2(j+1)}{k}+2\binom{n-2(j+1)}{k-1}+\binom{n-2(j+1)}{k-2}$ and letting $a_i=\binom{n-2(j+1)}{i-j}$ we obtain
    \begin{align*}
        c^{(i)}_{j,j+1}&=(a_{i-1}a_i-a_{i-2}a_{i+1})+4(a_{i-1}^2-a_{i-2}a_i)+(a_{i-2}a_{i-1}-a_{i-3}a_{i}) \geq 0,        
    \end{align*}
    where the last inequality follows from the log-concavity of the sequence $a_i$ and \Cref{lemma:log-concavity-general}.
\end{remark}

It seems hard to characterize when a coefficient $c_{jk}^{(i)}$ is negative. However, we prove \Cref{lem: negatives on the right} which describes a structure in the signs of such coefficients, provided that one makes special choices for the parameters. Fix some $\ell\geq 0$. We will show that among the coefficients $c_{jk}^{(i)}$ with $j+k=2\ell$ or $j+k=2\ell-1$, nonpositive ones can occur only at the ``tail". More precisely, writing such coefficients as a sequence $(c_{\ell,\ell}^{(i)},c_{\ell-1,\ell+1}^{(i)},c_{\ell-2,\ell+2}^{(i)},\dots)$ --- or similarly $(c_{\ell-1,\ell}^{(i)},c_{\ell-2,\ell+1}^{(i)},c_{\ell-3,\ell+2}^{(i)},\dots)$ --- we show that if this sequence contains a negative number, then all the ones following it must be nonpositive.

\begin{lemma}\label{lem: negatives on the right}
    Let $n\geq 1$, $0\leq i\leq  n/2$ and $1\leq \ell \leq (i+1)/2$. Then:
    \begin{enumerate}[\normalfont (i)]
        \item If $c_{\ell-j_0,\ell+j_0}^{(i)}<0$, then $c_{\ell-j,\ell+j}^{(i)}\leq 0$ for every $j\geq j_0$.
        \item If $c_{\ell-1-j_0,\ell+j_0}^{(i)}<0$, then $c_{\ell-1-j,\ell+j}^{(i)}\leq 0$ for every $j\geq j_0$.
    \end{enumerate}
\end{lemma}

\begin{proof}
    Let us write the proof for the first of the two sequences in detail, as the other is very similar. What we want to prove is equivalent to showing that for fixed $n$, $i$, $\ell$ as in the statement, the sequence $\{c^{(i)}_{\ell-j,\ell+j}\}_{j=1}^{\ell}$ can possibly become nonpositive only for sufficiently large values of $j$.

    By Lemma~\ref{lemma:coeffs}, for $j>0$,
        \begin{align}
            c_{\ell-j,\ell+j}^{(i)} 
        &= 2\cdot \tbinom{n-2\ell+2j}{i-\ell+j}\tbinom{n-2\ell-2j}{i-\ell-j} - 
        \tbinom{n-2\ell+2j}{i-\ell+j-1}\tbinom{n-2\ell-2j}{i-\ell-j+1}- 
        \tbinom{n-2\ell+2j}{i-\ell+j+1}\tbinom{n-2\ell-2j}{i-\ell-j-1}\label{eq:defining}\\
        &= \tbinom{n-2\ell+2j}{i-\ell+j}\tbinom{n-2\ell-2j}{i-\ell-j} \cdot \left(2 - \tfrac{(i-\ell+j)(n-\ell-j-i)}{(n-\ell+j-i+1)(i-\ell-j+1)} - \tfrac{(i-\ell-j)(n-\ell+j-i)}{(i-\ell+j+1)(n-\ell-j-i+1)}\right)\nonumber\\
        &= \tbinom{n-2\ell+2j}{i-\ell+j}\tbinom{n-2\ell-2j}{i-\ell-j} \cdot \tfrac{A j^2 + B}{(n-\ell+j-i+1)(i-\ell-j+1)(i-\ell+j+1)(n-\ell-j-i+1)},\label{eq:expression}
        \end{align}
    where $A$ and $B$ do not depend on $j$ and are given by:
    \begin{align*}
        A &:= -4n^2 + 16ni - 16i^2 - 2n + 4\ell - 2,\\
        B &:= 2n^2i - 2ni^2 - 2n^2\ell  - 4ni\ell + 4i^2\ell  + 6n\ell ^2 - 4\ell ^3 \\
        &\qquad + 2n^2  + 2ni - 2i^2 - 10n\ell  + 10\ell ^2 + 4n - 8\ell  + 2.
    \end{align*}
    We are interested in analyzing the sign of $c^{(i)}_{\ell-j,\ell+j}$ when $j\geq 1$. The crucial observation is that the behavior of this sign in the non-trivial cases is dictated only by $Aj^2+B$. Each of the factors in the denominator of the fraction appearing in \eqref{eq:expression} is nonnegative for the admissible values of $j$. Sometimes it can vanish, but precisely in those cases we already know what the sign of $c^{(i)}_{\ell-j,\ell+j}$ is. More precisely, taking, e.g., the term $i-\ell-j+1$ appearing in the denominator of \eqref{eq:expression}, we may safely assume that it is strictly positive, because if we had $\ell+j > i + 1$, then the number $c^{(i)}_{\ell-j,\ell+j}$ is actually zero, and similarly, if $\ell+j = i+1$, then $c^{(i)}_{\ell-j,\ell+j}<0$. 

    Now, we claim that $A < 0$ and $B > 0$. To prove the inequality of $A$, we can just write:
    \[ A = -4(n-2i)^2 - 2(n-2\ell) - 2\]
    and each summand is indeed negative. To see the sign of $B$, it suffices to replace $j=0$ in the expression for $c_{\ell-j,\ell+j}^{(i)}$ appearing in \eqref{eq:defining}---an important caveat, however, is that plugging $j=0$ will not give us $c^{(i)}_{\ell,\ell}$ because the case in which the two lower indices are equal is handled differently by Lemma~\ref{lemma:coeffs}. Nonetheless, after taking $j=0$ in that equation, the resulting value is:
        \[ 2 \left(\binom{n-2\ell}{i-\ell}^2 - \binom{n-2\ell}{i-\ell-1}\binom{n-2\ell}{i-\ell+1}\right),\]
    which is nonnegative by Remark~\ref{remark:cjji positive}. In particular, $B> 0$, because plugging $j=0$ in \eqref{eq:expression} leaves us with $B$ times a positive expression. 
    
    From all the above discussion, we have that the sign of $c^{(i)}_{\ell-j,\ell+j}$ is given by that $Aj^2+B$, and this expression is a quadratic having negative leading term, and evaluates to something nonnegative when $j=0$. Geometrically, this says that the graph of the function $j\mapsto Aj^2+B$ is a parabola that crosses the $j$-axis in exactly one positive value $j_0$ of $j$. In turn, this implies that $c^{(i)}_{\ell-j,\ell+j}$ will possibly become nonpositive for large $j$ within the valid range, and from a certain point on.
\end{proof}

\begin{example}\label{ex:n=16}
    Let us consider $n=16$, and $i=5$. The following is the resulting expression for $h_5^2-h_4h_6$:
    \begin{align*}
        h_5^2-h_4h_6 &= 4504864\gamma_0^2\\
        &\qquad +2186184 \gamma_0\gamma_1,\\
        &\qquad +273273 \gamma_1^2+492492\gamma_0\gamma_2\\
     &\qquad +128128 \gamma_1\gamma_2+94640 \gamma_0\gamma_3\\
     &\qquad +15730 \gamma_2^2+26390 \gamma_1\gamma_3+10920 \gamma_0\gamma_4\\
     &\qquad +6930 \gamma_2\gamma_3 + 3822 \gamma_1\gamma_4 -2184 \gamma_0\gamma_5\\
     &\qquad +825 \gamma_3^2 + 1177 \gamma_2\gamma_4 -182 \gamma_1\gamma_5 -1820 \gamma_0\gamma_6\\
     &\qquad +320 \gamma_3\gamma_4 + 44 \gamma_2\gamma_5 -364 \gamma_1\gamma_6 + 0 \gamma_0\gamma_7\\
     &\qquad +36\gamma_4^2 + 30 \gamma_3\gamma_5 -66\gamma_2\gamma_6 + 0 \gamma_1\gamma_7+0 \gamma_0\gamma_8.
    \end{align*}
    Let us focus on the terms $\gamma_j\gamma_k$ for which $j+k=6$, i.e., $\gamma_{3-\ell,3+\ell}$ for $\ell=0,\ldots,3$. From the above lemma, we know that the nonpositive terms (if any) must appear together and at the end. This is indeed the case, as our sequence is $(825,1177,-182,-1820)$. 
\end{example}

\section{The second step of the proof}

\begin{lemma}
    Let $a_0,\ldots,a_s$ be a sequence of numbers such that if $a_{j_0}<0$ for some index $j_0$, then $a_j\leq 0$ for all $j\geq j_0$. Let $b_0\geq \cdots \geq b_s\geq 0$ be a weakly decreasing sequence of nonnegative numbers. Then
        \[\sum_{j=0}^s a_j \geq 0 \qquad \Longrightarrow \qquad \sum_{j=0}^s a_jb_j \geq 0.\]
\end{lemma}

\begin{proof}
Let $A_i:=a_0+\cdots+a_i$ for $i=0,\ldots,s$ and set $A_{-1}:=0$, $A_{s+1}:=0$, $b_{s+1}:=0$.  We have that $a_i>0$ if and only if $A_i>A_{i-1}$, so if $j_0$ is the first index for which $a_{j_0}<0$, then $0\leq A_0 \leq A_1 \leq \cdots \leq A_{j_0-1} > A_{j_0} \geq \cdots \geq A_s \geq 0$ and thus $\{A_i\}$ is a unimodal sequence.  

Finally, observe that the following identity holds:
\[
\sum_{i=1}^s a_i b_i = \sum_{i=0}^s (A_i-A_{i-1})b_i = \sum_{i=0}^s A_i(b_i-b_{i+1}). 
\]
Since each term on the right hand side is nonnegative, the inequality follows. 
\end{proof}

The last lemma may be applied in the situation of Lemma~\ref{lem: negatives on the right}. Taking $a_j=c^{(i)}_{\ell-j,\ell+j}$ for each $j$, the condition on the signs is satisfied. Taking $b_j = \gamma_{\ell-j}\gamma_{\ell+j}$ for each $j$ gives a weakly decreasing sequence, thanks to Lemma~\ref{lemma:log-concavity-general}. In particular, we have that

\[ \sum_{j\geq 0} c^{(i)}_{\ell-j,\ell+j} \geq 0 \qquad \Longrightarrow \qquad \sum_{j\geq 0} c^{(i)}_{\ell-j,\ell+j} \gamma_{\ell-j}\gamma_{\ell+j}\geq 0.\]
The condition on the right hand side of the above display is exactly what we need to prove in order to conclude that in $h_i^2-h_{i-1}h_{i+1}$ the contribution of all the products $\gamma_j\gamma_k$ with fixed $j+k=2\ell$ is nonnegative. An analogous reasoning applies to the case of $c^{(i)}_{\ell-1-j,\ell+j}$. In summary, we have reduced the problem to showing that the sum on the left and, similarly, the odd counterpart $\sum_{j\geq 0} c_{\ell-1-j,\ell+j}$, are nonnegative.

Our next goal is to show that, indeed, the two sums
\begin{equation} \label{eq:two-sums}
\sum_{j\geq 0} c_{\ell-j,\ell+j}^{(i)} \qquad \text{ and } \qquad \sum_{j\geq 0} c_{\ell-1-j,\ell+j}^{(i)}\end{equation}
are nonnegative. These nonnegativity properties are equivalent to proving the following neat statement, which we will show in the next subsection.

\begin{theorem}\label{thm:reduced-case}
    Let $n\in\mathbb{N}$ and let $u$ be a formal variable, and let $\gamma_h(x) = \sum_{j=0}^{\lfloor n/2\rfloor} u^jx^j$. Then 
    \[
        [u^r](h_i^2-h_{i-1}h_{i+1})\geq 0,
    \]
    for every $1\leq i\leq \lfloor\frac{n}{2}\rfloor$ and every $r\geq 0$.
\end{theorem}

The equivalence is immediate, as we are replacing each $\gamma_k$ by $u^k$, so that for each index $j$ we have that $\gamma_{\ell-j}\gamma_{\ell+j} = u^{2\ell}$, and hence the coefficient of $u^{2\ell}$ in $h_i^2-h_{i-1}h_{i+1}$ is precisely the sum appearing on the left in \eqref{eq:two-sums} and, similarly, the coefficient of $u^{2\ell-1}$ is the sum on the right. In other words, in light of all the preceding discussion, if we show Theorem~\ref{thm:reduced-case}, the proof of Theorem~\ref{thm:main} follows.

\subsection*{The proof of Theorem~\ref{thm:reduced-case} via lattice paths}

This subsection is devoted to the proof of Theorem~\ref{thm:reduced-case}. Even though the statement is very simple, the proof we have been able to produce resorts to the enumeration of a complicated family of lattice paths. 

First, we make the inequalities to prove more explicit. We have that 
    \[
        h(x;u) = \sum_{j=0}^{\lfloor n/2\rfloor} u^jx^j(1+x)^{n-2j} = \sum_{i=0}^{\lfloor n/2\rfloor}x^i\sum_{j=0}^i\binom{n-2j}{i-j}u^j.
    \]
    Thus,
    \begin{align*}
        h_i^2-h_{i-1}h_{i+1}= \sum_{j,k=0}^{i} \binom{n-2j}{i-j}\binom{n-2k}{i-k}u^{j+k} - \sum_{j,k=0}^{i+1} \binom{n-2j}{i-1-j}\binom{n-2k}{i+1-k}u^{j+k}.
    \end{align*}

We now show that this quantity is nonnegative.

\begin{theorem}\label{thm: sum nonnegative}
    Let $n\geq 0$ and let $i\leq \lfloor\frac{n}{2}\rfloor$. Then
    \begin{align}\label{eq: lhs rhs}
        \sum_{\substack{j,k=0\\ j+k = r}}^{i}\binom{n-2j}{i-j}\binom{n-2k}{i-k}\geq \sum_{\substack{j,k=0\\ j+k = r}}^{i+1}\binom{n-2j}{i-1-j}\binom{n-2k}{i+1-k},
    \end{align}
    for every $r\geq 0$.
\end{theorem}

\begin{proof}
    We employ a path-counting argument coupled with an injection. First recall that for integers $a\leq c$ and $b\leq d$ we have that
    \[
        \#\{\text{paths from $(a,b)$ to $(c,d)$}\} = \binom{c+d-(a+b)}{d-b},
    \]
    where a path is meant to be a \emph{north-east} lattice path, i.e., a lattice path using steps $+(1,0)$ or $+(0,1)$.
    
    Fix now $r\geq 0$ and consider the lattice point $D = (2n-2i-r,2i-r)$.
    Consider first the left hand side (LHS) in \eqref{eq: lhs rhs}. Let $P = (n-2i,0)$ and $Q = (n-i,i)$. For a lattice point $A=(n-i-j,i-j) \in PQ$ we observe that 
    \begin{multline*}
        \#\{\text{paths from $O$ to $A$}\}\cdot \#\{\text{paths from $A$ to $D$}\}=\\
        \binom{n-2j}{i-j}\binom{2n-2r-(n-2j)}{2i-r-(i-j)}=
        \binom{n-2j}{i-j}\binom{n-2k}{i-k},
    \end{multline*}
    where in the last identity we used that $j+k=r$. In other words, the summands in the LHS of \eqref{eq: lhs rhs} are the product of number of paths from $O$ to $A$ and from $A$ to $D$ as $A$ varies in the segment $PQ$. These are the paths contained in the yellow rectangles depicted in \Cref{fig: proof of 4.3}. By a double counting argument we can compute the LHS also by counting the paths from $O$ to $D$, where each path contributes for the number of lattice points in $PQ$ it intersects. More precisely:
    \begin{align*}
        \text{LHS} &= \sum_{A\in PQ\cap\mathbb{Z}^2} \#\{\text{paths from $O$ to $A$}\}\#\{\text{paths from $A$ to $D$}\}\\
        &=\sum_{\alpha \text{ path from $O$ to $D$}}\sum_{A \in \alpha\cap PQ} 1\\
        &=\sum_{\alpha \text{ path from $O$ to $D$}}\#\{\alpha \cap PQ\}
    \end{align*}
    \begin{figure}[h]
        \centering
        \includegraphics[scale=0.9]{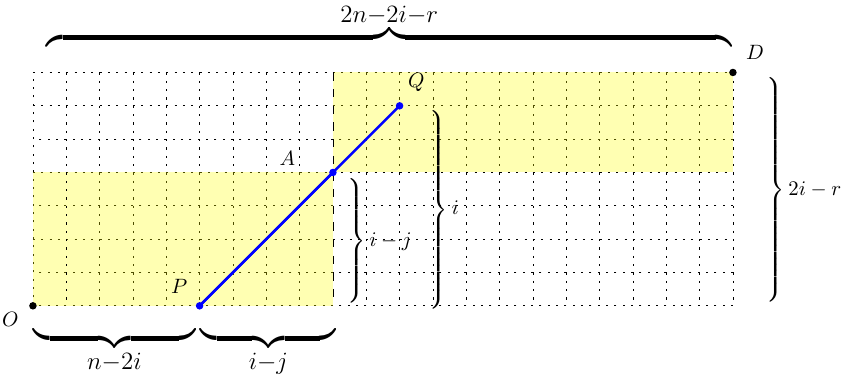}
        \caption{Each summand in the LHS of \eqref{eq: lhs rhs} is the product of the number of paths contained in the yellow rectangles as $A$ varies in $PQ$.}
        \label{fig: proof of 4.3}
    \end{figure}

    Now, let us consider the right hand side (RHS) of \eqref{eq: lhs rhs}. We employ a similar strategy. Let $P'=(n-2i+2,0)$ and $Q' =(n-i+1,i-1)$. For any lattice point $A'=(n-i+1-j,i-1-j) \in P'Q'$ we have that 
    \begin{multline*}
        \#\{\text{paths from $O$ to $A'$}\}\#\{\text{paths from $A'$ to $D$}\}=\\
        \binom{n-2j}{i-1-j}\binom{2n-2r-(n-2j)}{2i-r-(i-1-j)}=
        \binom{n-2j}{i-1-j}\binom{n-2k}{i+1-k}.
    \end{multline*}
    See \Cref{fig: proof 4.3 P'} for an illustration. In particular,
    \begin{align*}
        \text{RHS} &= \sum_{A'\in P'Q'\cap\mathbb{Z}^2} \#\{\text{paths from $O$ to $A'$}\}\#\{\text{paths from $A'$ to $D$}\}\\
        &=\sum_{\alpha \text{ path from $O$ to $D$}}\sum_{A' \in \alpha\cap P'Q'} 1\\
        &=\sum_{\alpha \text{ path from $O$ to $D$}}\#\{\alpha \cap P'Q'\}.
    \end{align*}
    \begin{figure}[h]
        \centering
        \includegraphics[scale=0.9]{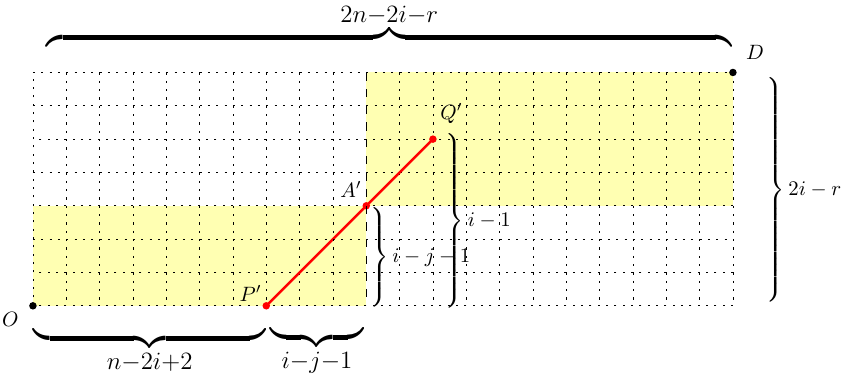}
        \caption{Each summand in the RHS of \eqref{eq: lhs rhs} is the product of the number of paths contained in the yellow rectangles as $A'$ varies in $P'Q'$.}
        \label{fig: proof 4.3 P'}
    \end{figure}
    \begin{claim}\label{claim}
        If $\alpha$ is a path from $O$ to $D$ such that $\alpha\cap P'Q'\neq \varnothing$, then $\alpha\cap PQ\neq \varnothing$.
    \end{claim}
    \begin{proof}[Proof of \Cref{claim}]
        Let $B'=(x',y')\in \alpha\cap P'Q'$. In particular, $x'-y'=n-2i+2$ and $0\leq y'\leq i-1$. Walking backwards (from $B'$ to $O$) in $\alpha$ the difference between the first and second coordinate of a point either increases or decreases by one, and the second coordinate weakly decreases. Since $O$ is in $\alpha$, the difference between the coordinates will eventually be $0$. By continuity, there exists a lattice point $B=(x,y)\in \alpha$ such that $x-y = n-2i$ and $0\leq y\leq i-1$. So $B\in \alpha\cap PQ$.
    \end{proof}
    Observe that the natural lattice partial order in $\mathbb{Z}^2$, which we denote with $\leq$ induces a total order in $PQ\cap \mathbb{Z}^2$ and $P'Q'\cap \mathbb{Z}^2$. If $\alpha$ is a path from $O$ to $D$ such that $\alpha\cap P'Q'\neq \varnothing$ we denote by $\max\{\alpha\cap P'Q'\}$ the maximal element with respect to this order. By \Cref{claim} we have that $\alpha\cap PQ\neq \varnothing$, and we denote by $\min\{\alpha\cap PQ\}$ its minimal element. Since in the proof of \Cref{claim} we showed that given $B'\in \alpha\cap P'Q'$ there exists $B\in \alpha\cap PQ$ with $B\leq B'$, we have that $\min\{\alpha\cap PQ\}\leq \max\{\alpha\cap P'Q'\}$. We will refine the path counting using these elements in the following way: any path $\alpha$ from $O$ to $D$ which intersects $P'Q'$ can be written uniquely as the concatenation of a path $\alpha_1$ from $O$ to $\min\{\alpha\cap PQ\}$, a path $\alpha_2$ from $\min\{\alpha\cap PQ\}$ to $\max\{\alpha\cap P'Q'\}$ and a path $\alpha_3$ from $\max\{\alpha\cap P'Q'\}$ to $D$. In terms of lattice point count we have that 
    \begin{equation}\label{eq: alpha cap PQ}
        \{\alpha \cap PQ\}=\{\alpha_2 \cap PQ\}\cup \{\alpha_3 \cap PQ\}     
    \end{equation}
    and 
    \begin{equation}\label{eq: alpha cap P'Q'}
        \{\alpha \cap P'Q'\}=\{\alpha_2 \cap P'Q'\}.
    \end{equation}
    The next claim shows that the sum of the contributions of the intermediate paths vanishes.
    \begin{claim}\label{claim2}
        Let $R\in PQ\cap\mathbb{Z}^2$ and $R'\in P'Q'\cap\mathbb{Z}^2$, with $R\leq R'$. Then 
        \[
        \sum_{\alpha \text{ path $R\to R'$}}\#\{\alpha \cap PQ\}=\sum_{\alpha \text{ path $R\to R'$}}\#\{\alpha \cap P'Q'\}.
        \]
    \end{claim}
    \begin{proof}[Proof of \Cref{claim2}]
        The rotation by $180^{\circ}$ around the barycenter of the rectangle spanned by $R$ and $R'$ induces an involution $\varphi$ on the set of paths from $R$ to $R'$. Moreover, this map sends lattice points of $PQ$ to lattice points of $P'Q'$. See \Cref{fig: involution} for an illustration of the map $\varphi$, showing how the three lattice points in the intersection of the path on the left with $PQ$ (the blue segment) correspond to three lattice points in the intersection of the path on the right with $P'Q'$ (the red segment). We obtain 
        \begin{align*}
            \sum_{\alpha \text{ path $R\to R'$}}\#\{\alpha \cap PQ\}&=\sum_{\alpha \text{ path $R\to R'$}}\#\{\varphi(\alpha) \cap PQ\}\\
            &=\sum_{\alpha \text{ path $R\to R'$}}\#\{\alpha \cap P'Q'\}.\qedhere
        \end{align*}
    \end{proof}
    \begin{figure}[h]
        \centering
        \includegraphics[scale=0.9]{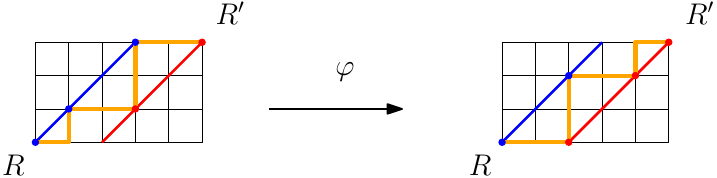}
        \caption{An illustration of the involution in the proof of \Cref{claim2}.}
        \label{fig: involution}
    \end{figure}
    Finally, we can show that the difference $\text{LHS} - \text{RHS}$ is nonnegative:
    \allowdisplaybreaks
    \begin{align*}
        \text{LHS} - \text{RHS} &= \sum_{\alpha \text{ path $O\to D$}}\left(\#\{\alpha \cap PQ\}-\#\{\alpha \cap P'Q'\}\right)\\
        &=\sum_{\substack{\alpha \text{ path $O\to D$}\\
        \alpha\cap P'Q'= \varnothing}}\#\{\alpha \cap PQ\}
       +\sum_{\substack{\alpha \text{ path $O\to D$}\\
        \alpha\cap P'Q'\neq \varnothing}}\left(\#\{\alpha \cap PQ\}-\#\{\alpha \cap P'Q'\}\right)\\
        &=\sum_{\substack{\alpha \text{ path $O\to D$}\\
        \alpha\cap P'Q'= \varnothing}}\#\{\alpha \cap PQ\}\\
        &\qquad +\sum_{\substack{R\in PQ\cap\mathbb{Z}^2\\ R'\in P'Q'\cap\mathbb{Z}^2\\
        R\leq R'}}\left(\sum_{\substack{\alpha \text{ path $O$ to $D$}\\ \min\{\alpha\cap PQ\}=R\\ \max\{\alpha\cap P'Q'\}=R'} }\left(\#\{\alpha \cap PQ\}-\#\{\alpha \cap P'Q'\}\right)\right).
\end{align*}

We now combine \eqref{eq: alpha cap PQ} and \eqref{eq: alpha cap P'Q'}:
  \begin{multline}
        \text{LHS} - \text{RHS}=\sum_{\substack{\alpha \text{ path $O\to D$}\\
        \alpha\cap P'Q'= \varnothing}}\#\{\alpha \cap PQ\} \enspace +\\
        \sum_{\substack{R\in PQ\cap\mathbb{Z}^2\\ R'\in P'Q'\cap\mathbb{Z}^2\\
        R\leq R'}}\left[\sum_{\alpha \text{ path $R$ to $R'$}}\left(\#\{\alpha \cap PQ\}-\#\{\alpha \cap P'Q'\}\right)+ \sum_{\substack{\alpha \text{ path $R'$ to $D$}\\ \alpha\cap P'Q'=\{R'\}}}\#\{\alpha \cap PQ\} \right]\\
        =\sum_{\substack{\alpha \text{ path $O\to D$}\\
        \alpha\cap P'Q'= \varnothing}}\#\{\alpha \cap PQ\}+\sum_{\substack{R\in PQ\cap\mathbb{Z}^2\\ R'\in P'Q'\cap\mathbb{Z}^2\\
        R\leq R'}}\sum_{\substack{\alpha \text{ path $R'$ to $D$}\\ \alpha\cap P'Q'=\{R'\}}}\#\{\alpha \cap PQ\},\label{eq: final sums}  
 \end{multline}
where the last equality follows from \Cref{claim2}. In particular, $\text{LHS} - \text{RHS}\geq 0$.
\end{proof}

\begin{example}
    Let $n=6$, $i=2$ and $r=2$. By \Cref{ex:gamma-degree-3} we know that $c_{1,1}^{(2)}+c_{0,2}^{(2)}=28$ and, by \eqref{eq: final sums} in the proof of \Cref{thm: sum nonnegative}, we know that this number is equal to
    \[
        \sum_{\substack{\alpha \text{ path $O\to D$}\\
        \alpha\cap P'Q'= \varnothing}}\#\{\alpha \cap PQ\}+\sum_{\substack{R\in PQ\cap\mathbb{Z}^2\\ R'\in P'Q'\cap\mathbb{Z}^2\\
        R\leq R'}}\sum_{\substack{\alpha \text{ path $R'$ to $D$}\\ \alpha\cap P'Q'=\{R'\}}}\#\{\alpha \cap PQ\},
    \]
    with $D=(6,2)$, $P=(2,0)$, $Q=(4,2)$, $P'=(4,0)$ and $Q'=(5,1)$. \Cref{fig:paths} shows the 15 paths which give a non-trivial contribution to these sums. Fourteen of them do not intersect $P'Q'$ (the red segment), and each contributes to the first sum with the number of intersections with $PQ$ (the blue segment). The sum of these numbers is 27. One of them, the one in the third row and first column, contributes with a 1 to the second sum: indeed it contains a subpath from $R'=P'$ to $D$ which intersects $PQ$. Hence the contributions sum up to $28$.
\end{example}

\begin{figure}
    \centering
    \begin{tikzpicture}[scale=0.6]
    
    \draw[step=1cm] (0,0) grid (6,2);
    
    \node[below left] at (0,0) {$O$};
    \node[below] at (2,0) {$P$};
    \node[below] at (4,0) {$P'$};
    \node[above] at (4,2) {$Q$};
    \node[above right] at (5,1) {$Q'$};
    \node[above right] at (6,2) {$D = (6,2)$};
    
    \draw[thick,blue] (2,0) -- (4,2);
    
    \draw[thick,red] (4,0) -- (5,1);
    \fill[blue] (2,0) circle (2pt); 
    \fill[blue] (4,2) circle (2pt); 
    \fill[red] (4,0) circle (2pt);  
    \fill[red] (5,1) circle (2pt);  
    
    \end{tikzpicture}
\end{figure}
\begin{figure}
    \begin{tikzpicture}[scale=0.38]
    
    \draw[step=1cm] (0,0) grid (6,2);
        
    \draw[thick,blue] (2,0) -- (4,2);
    
    \draw[thick,red] (4,0) -- (5,1);

    \draw[very thick,line cap=round,Orange] (0,0)--(0,2) -- (6,2);
    \end{tikzpicture}\quad
    \begin{tikzpicture}[scale=0.38]
    
    \draw[step=1cm] (0,0) grid (6,2);
        
    \draw[thick,blue] (2,0) -- (4,2);
    
    \draw[thick,red] (4,0) -- (5,1);

    \draw[very thick,line cap=round,Orange] (0,0)--(0,1) -- (1,1) -- (1,2) -- (6,2);
    
    \end{tikzpicture}\quad
    \begin{tikzpicture}[scale=0.38]
    
    \draw[step=1cm] (0,0) grid (6,2);
        
    \draw[thick,blue] (2,0) -- (4,2);
    
    \draw[thick,red] (4,0) -- (5,1);

    \draw[very thick,line cap=round,Orange] (0,0)--(0,1) -- (2,1) -- (2,2) -- (6,2);
    \end{tikzpicture}\quad
    \begin{tikzpicture}[scale=0.38]
    
    \draw[step=1cm] (0,0) grid (6,2);
        
    \draw[thick,blue] (2,0) -- (4,2);
    
    \draw[thick,red] (4,0) -- (5,1);

    \draw[very thick,line cap=round,Orange] (0,0)--(1,0) -- (1,1) -- (2,1) -- (2,2) -- (6,2);
    \end{tikzpicture}\quad
    \begin{tikzpicture}[scale=0.38]
    
    \draw[step=1cm] (0,0) grid (6,2);
        
    \draw[thick,blue] (2,0) -- (4,2);
    
    \draw[thick,red] (4,0) -- (5,1);

    \draw[very thick,line cap=round,Orange] (0,0)--(1,0) -- (1,2) -- (6,2);
    \end{tikzpicture}\\
    \vspace{0.5cm}
    \begin{tikzpicture}[scale=0.38]
    
    \draw[step=1cm] (0,0) grid (6,2);
        
    \draw[thick,blue] (2,0) -- (4,2);
    
    \draw[thick,red] (4,0) -- (5,1);

    \draw[very thick,line cap=round,Orange] (0,0)--(2,0) -- (2,2) -- (6,2);
    \end{tikzpicture}\quad
    \begin{tikzpicture}[scale=0.38]
    
    \draw[step=1cm] (0,0) grid (6,2);
        
    \draw[thick,blue] (2,0) -- (4,2);
    
    \draw[thick,red] (4,0) -- (5,1);

    \draw[very thick,line cap=round,Orange] (0,0)--(0,1) -- (3,1) -- (3,2) -- (6,2);
    \end{tikzpicture}\quad
    \begin{tikzpicture}[scale=0.38]
    
    \draw[step=1cm] (0,0) grid (6,2);
        
    \draw[thick,blue] (2,0) -- (4,2);
    
    \draw[thick,red] (4,0) -- (5,1);

    \draw[very thick,line cap=round,Orange] (0,0)--(0,1) -- (4,1) -- (4,2) -- (6,2);
    \end{tikzpicture}\quad
    \begin{tikzpicture}[scale=0.38]
    
    \draw[step=1cm] (0,0) grid (6,2);
        
    \draw[thick,blue] (2,0) -- (4,2);
    
    \draw[thick,red] (4,0) -- (5,1);

    \draw[very thick,line cap=round,Orange] (0,0)--(1,0) -- (1,1) -- (3,1) -- (3,2) -- (6,2);
    \end{tikzpicture}\quad
    \begin{tikzpicture}[scale=0.38]
    
    \draw[step=1cm] (0,0) grid (6,2);
        
    \draw[thick,blue] (2,0) -- (4,2);
    
    \draw[thick,red] (4,0) -- (5,1);

    \draw[very thick,line cap=round,Orange] (0,0)--(0,1) -- (1,1) -- (4,1) -- (4,2) -- (6,2);
    \end{tikzpicture}\\
    \vspace{0.5cm}
    \begin{tikzpicture}[scale=0.38]
    
    \draw[step=1cm] (0,0) grid (6,2);
        
    \draw[thick,blue] (2,0) -- (4,2);
    
    \draw[thick,red] (4,0) -- (5,1);

    \draw[very thick,line cap=round,Orange] (0,0)--(4,0) -- (4,2) -- (6,2);
    \end{tikzpicture}\quad
    \begin{tikzpicture}[scale=0.38]
    
    \draw[step=1cm] (0,0) grid (6,2);
        
    \draw[thick,blue] (2,0) -- (4,2);
    
    \draw[thick,red] (4,0) -- (5,1);

    \draw[very thick,line cap=round,Orange] (0,0)--(2,0) -- (2,1) -- (3,1) -- (3,2) -- (6,2);
    \end{tikzpicture}\quad
    \begin{tikzpicture}[scale=0.38]
    
    \draw[step=1cm] (0,0) grid (6,2);
        
    \draw[thick,blue] (2,0) -- (4,2);
    
    \draw[thick,red] (4,0) -- (5,1);

    \draw[very thick,line cap=round,Orange] (0,0)--(3,0) -- (3,2) -- (6,2);
    \end{tikzpicture}\quad
    \begin{tikzpicture}[scale=0.38]
    
    \draw[step=1cm] (0,0) grid (6,2);
        
    \draw[thick,blue] (2,0) -- (4,2);
    
    \draw[thick,red] (4,0) -- (5,1);

    \draw[very thick,line cap=round,Orange] (0,0)--(2,0) -- (2,1) -- (4,1) -- (4,2) -- (6,2);
    \end{tikzpicture}\quad
    \begin{tikzpicture}[scale=0.38]
    
    \draw[step=1cm] (0,0) grid (6,2);
        
    \draw[thick,blue] (2,0) -- (4,2);
    
    \draw[thick,red] (4,0) -- (5,1);

    \draw[very thick,line cap=round,Orange] (0,0)--(3,0) -- (3,1) -- (4,1) -- (4,2) -- (6,2);
    \end{tikzpicture}\caption{Lattice paths which contribute to $[u^2](h_2^2-h_1h_3)$ with $n=6$.}\label{fig:paths}
\end{figure}

\begin{remark}
    When $r\geq i+1$ the second coordinate of $Q'$ is larger than or equal to the second coordinate of $D$. In other words $Q'$ is ``above" $D$. In this case, each path from $O$ to $D$ intersects $P'Q'$. For the same reason, we have that $\max\{\alpha\cap P'Q'\}$ is larger than any point in $\alpha\cap PQ$, and hence there is no non-zero contribution to the second sum of \eqref{eq: final sums}. Following the proof of \Cref{thm: sum nonnegative} this implies that $\sum_{\substack{j,k=0\\i+j=r}}^i{c_{jk}^{(i)}}=0$. We observe that there is another way to see this identity when $r\geq i+1$ using generating functions. Indeed, by \cite[Eq.~(2.5.15)]{generatingfunctionology} we have that
    \[
    \binom{n-2j}{i-j}=\binom{2(i-j)+(n-2i)}{i-j} = [x^{n-i-j}]\frac{(1-\sqrt{1-4x})^{n-2i}}{2\sqrt{1-4x}},
    \]
    where the square brackets indicate the coefficient in the power series expansion.
    When $r\geq i+1$, both the LHS and the RHS of \eqref{eq: lhs rhs} are equal to 
    \[[x^{2n-2i-r}] \left(\frac{(1-\sqrt{1-4x})^{n-2i}}{2\sqrt{1-4x}}\right)^2.\]
    
    Moreover, the computation in \eqref{eq: final sums} implies that it is possible to express the difference between both sides in \eqref{eq: lhs rhs} as a manifestly positive expression involving binomial coefficients. An alternative approach that performs this computation via the Wilf--Zeilberger method was privately communicated to us by Tewodros Amdeberhan.
\end{remark}

\section*{Acknowledgments}

The authors thank two anonymous reviewers for their careful reading and thoughtful suggestions. Luis Ferroni is supported by the Minerva Research Foundation at the Institute for Advanced Study. Greta Panova is partially supported by an NSF CCF grant and a Simons Fellowship.  Lorenzo Venturello is partially supported by the PRIN project ``Algebraic and geometric aspects of Lie theory" (2022S8SSW2) and by INdAM -- GNSAGA. 

\bibliographystyle{amsalpha}
\bibliography{bibliography}

\providecommand{\bysame}{\leavevmode\hbox to3em{\hrulefill}\thinspace}
\providecommand{\MR}{\relax\ifhmode\unskip\space\fi MR }
\providecommand{\MRhref}[2]{%
  \href{http://www.ams.org/mathscinet-getitem?mr=#1}{#2}
}
\providecommand{\href}[2]{#2}
\begin{thebibliography}{BFJ24}

\bibitem[Ath18]{athanasiadis-gamma-positivity}
Christos~A. Athanasiadis, \emph{Gamma-positivity in combinatorics and geometry}, S\'{e}m. Lothar. Combin. \textbf{77} ([2016--2018]), Art. B77i, 64. \MR{3878174}

\bibitem[BFJ24]{branden-ferroni-jochemko}
Petter {Br{\"a}nd{\'e}n}, Luis {Ferroni}, and Katharina {Jochemko}, \emph{{Preservation of inequalities under Hadamard products}}, arXiv e-prints (2024), arXiv:2408.12386.

\bibitem[BH20]{branden-huh}
Petter Br{\"a}nd{\'e}n and June Huh, \emph{Lorentzian polynomials}, Ann. of Math. (2) \textbf{192} (2020), no.~3, 821--891. \MR{4172622}

\bibitem[Br{\"a}06]{branden-gamma}
Petter Br{\"a}nd{\'e}n, \emph{Sign-graded posets, unimodality of {$W$}-polynomials and the {C}harney-{D}avis conjecture}, Electron. J. Combin. \textbf{11} (2006), no.~2, Research Paper 9, 15. \MR{2120105}

\bibitem[Br{\"a}15]{branden}
Petter Br{\"a}{n}d{\'e}n, \emph{Unimodality, log-concavity, real-rootedness and beyond}, Handbook of enumerative combinatorics, Discrete Math. Appl. (Boca Raton), CRC Press, Boca Raton, FL, 2015, pp.~437--483. \MR{3409348}

\bibitem[Bre94]{brenti}
Francesco Brenti, \emph{Log-concave and unimodal sequences in algebra, combinatorics, and geometry: an update}, Jerusalem combinatorics '93, Contemp. Math., vol. 178, Amer. Math. Soc., Providence, RI, 1994, pp.~71--89. \MR{1310575}

\bibitem[Gal05]{gal}
{{\'S}wiatoslaw}~R. Gal, \emph{Real root conjecture fails for five- and higher-dimensional spheres}, Discrete Comput. Geom. \textbf{34} (2005), no.~2, 269--284. \MR{2155722}

\bibitem[Sta89]{stanley-unimodality}
Richard~P. Stanley, \emph{Log-concave and unimodal sequences in algebra, combinatorics, and geometry}, Graph theory and its applications: {E}ast and {W}est ({J}inan, 1986), Ann. New York Acad. Sci., vol. 576, New York Acad. Sci., New York, 1989, pp.~500--535. \MR{1110850}

\bibitem[Wil94]{generatingfunctionology}
Herbert~S. Wilf, \emph{generatingfunctionology}, second ed., Academic Press, Inc., Boston, MA, 1994. \MR{1277813}

\end{thebibliography}

\end{document}